\documentclass[svgnames]{amsart}
\usepackage{
	mathtools,
	tikz-cd,
	hyperref,
	cleveref,
}

\hypersetup{colorlinks, allcolors=MediumBlue}

\crefdefaultlabelformat{#2\textup{#1}#3} 

\newtheorem{thm}  {Theorem} [section]
\newtheorem{cri} [thm] {Criterion}

\newtheorem{lem} [thm] {Lemma}

\theoremstyle{definition} 
\newtheorem{dfn} [thm] {Definition}
\newtheorem{rmk} [thm] {Remark}
\newtheorem{exm} [thm] {Example}
\newtheorem*{prob} {Problem}

\numberwithin{equation}{section}

\crefname{thm}{Theorem}{Theorems}
\crefname{lem}{Lemma}{Lemmas}
\crefname{prop}{Proposition}{Propositions}
\crefname{dfn}{Definition}{Definitions}
\crefname{rmk}{Remark}{Remarks}
\crefname{exm}{Example}{Examples}
\crefname{cri}{Criterion}{Criteria}

\DeclarePairedDelimiterX\set[1]\lbrace\rbrace{\,\def\given{\;\delimsize\vert\;}#1\,} 

\DeclareMathOperator{\Epi}{Epi}
\DeclareMathOperator{\Hom}{Hom}
\DeclareMathOperator{\im}{im}
\newcommand{\un}{\mathrm{un}}
\newcommand{\FF}{\mathbb{F}}
\newcommand{\ZZ}{\ensuremath{\mathbb{Z}}}
\newcommand{\QQ}{\ensuremath{\mathbb{Q}}}
\newcommand{\cK}{\ensuremath{\mathcal{K}}}
\newcommand{\cH}{\ensuremath{\mathcal{H}}}
\newcommand{\size}[1]{\ensuremath{\lvert #1 \rvert}}

\begin{document}

\title[The isomorphism problem for group algebras]{The isomorphism problem for group algebras:\\ a criterion}
\author[T.~Sakurai]{\href{https://orcid.org/0000-0003-0608-1852}{Taro Sakurai}}
\address{Department of Mathematics and Informatics, Graduate School of Science, Chiba University, 1-33, Yayoi-cho, Inage-ku, Chiba-shi, Chiba, 263-8522 Japan}
\email{tsakurai@math.s.chiba-u.ac.jp}

\subjclass[2010]{%
	\href{https://zbmath.org/classification/?q=20C05}{20C05} (%
	\href{https://zbmath.org/classification/?q=16N20}{16N20},
	\href{https://zbmath.org/classification/?q=16U60}{16U60},
	\href{https://zbmath.org/classification/?q=20D15}{20D15},
	\href{https://zbmath.org/classification/?q=20C20}{20C20})}
\date{\today}
\begin{abstract}
	\noindent 
	Let $R$ be a finite unital commutative ring.
	We introduce a new class of finite groups, which we call \emph{hereditary groups} over~$R$.
	Our main result states that if $G$ is a hereditary group over $R$ then a unital algebra isomorphism between group algebras $RG \cong RH$ implies a group isomorphism $G \cong H$ for every finite group $H$.

	As application, we study the modular isomorphism problem, which is the isomorphism problem for finite $p$-groups over $R = \mathbb{F}_p$ where $\mathbb{F}_p$ is the field of $p$~elements.
	We prove that a finite $p$-group~$G$ is a hereditary group over $\mathbb{F}_p$ provided $G$ is abelian, $G$ is of class two and exponent~$p$ or $G$ is of class two and exponent four.
	These yield new proofs for the theorems by Deskins and Passi-Sehgal.
\end{abstract}

\maketitle

\tableofcontents

\section*{Introduction}
\noindent 
Let \( R \) be a unital commutative ring\footnote{Throughout this paper, we do not impose a ring (or algebra) to have~\( 1 \) but we impose \( 0 \neq 1 \) if it has~\( 1 \).
} and \( G \) a finite group.
Structure of the group algebra~\( RG \) of~\( G \) over~\( R \) reflects structure of the group~\( G \) to some extent.
The \emph{isomorphism problem} asks whether a group algebra~\( RG \) determines the group~\( G \) under a different setting.
See a survey by Sandling~\cite{San85} for a historical account for this problem.
In what follows, \( p \) stands for a prime number and \( \FF_p \) denotes the field of \( p \)~elements.
The only classic problem that is still open for more than 60 years is the \emph{modular isomorphism problem}. 
\begin{prob}
	Let \( G \) and \( H \) be finite \( p \)-groups.
	Does a unital algebra isomorphism \( \FF_pG \cong \FF_pH \) imply a group isomorphism \( G \cong H \)?
\end{prob}
Several positive solutions for special classes of \( p \)-groups are known;
See lists in the introduction of \cite{EK11} or \cite{HS06}.

In this paper, we introduce a new class of finite groups, which we call \emph{hereditary groups} over a finite unital commutative ring~\( R \) (\cref{dfn:hereditary group}).
Our main result (\cref{cri:iso}) states that if \( G \) is a hereditary group over~\( R \) then a unital algebra isomorphism \( RG \cong RH \) implies a group isomorphism \( G \cong H \) for every finite group~\( H \).
The proof rests on counting homomorphisms (\cref{lem:hom}) and adjoint (\cref{lem:adjoint})---this indirect approach is novel in the sense that it does not involve normalized isomorphisms, for example.
As application, we study the modular isomorphism problem.
We prove that a finite \( p \)-group~\( G \) is a hereditary group over \( \FF_p \) provided \( G \) is abelian (\cref{lem:abelian p-groups}), \( G \) is of class two and exponent~\( p \) (\cref{lem:exp p}) or \( G \) is of class two and exponent four (\cref{lem:exp 4}).
These yield new proofs for the theorems by Deskins (\cref{thm:Deskins}) and Passi-Sehgal (\cref{thm:Passi-Sehgal-odd,thm:Passi-Sehgal-even}) which are early theorems on the modular isomorphism problem.

\section{Criterion}
This section is devoted to proving our main result (\cref{cri:iso}).
The first lemma is an easy application of the inclusion-exclusion principle.
\begin{lem}
	\label{lem:epi}
	Let \( G \) and \( H \) be finite groups.
	We denote by \( \Epi(G, H) \) the set of all epimorphisms from~\( G \) to~\( H \).
	Let \( \cH \) be the set of all maximal subgroups of~\( H \).
	Then
	\begin{equation*}
		\size{\Epi(G, H)} =
		\sum_{\cK \subseteq \cH} (-1)^{\size{\cK}} \size{\Hom(G, \bigcap_{K \in \cK} K)}.
	\end{equation*}
\end{lem}

\begin{proof}
	First, note that
	\begin{align*}
		\Epi(G, H) &= \set{ f \colon G \to H \given \im f = H } \\
			&= \bigcap_{K \in \cH} \set{ f \colon G \to H \given \im f \not\le K }.
	\end{align*}
	Thus, by letting \( \Hom^K(G, H) = \set{ f \colon G \to H \given \im f \le K } \), it becomes
	\begin{equation*}
		\Epi(G, H) = \Hom(G, H) - \bigcup_{K \in \cH} \Hom^K(G, H).
	\end{equation*}
	By the inclusion-exclusion principle, we have
	\begin{align*}
		\size{\Epi(G, H)} &= \size{\Hom(G, H)} - \size{\bigcup_{K \in \cH} \Hom^K(G, H)} \\
		&= \size{\Hom(G, H)} + \sum_{\emptyset \neq \cK \subseteq \cH} (-1)^{\size{\cK}} \size{\bigcap_{K \in \cK} \Hom^K(G, H)} \\
		&= \size{\Hom(G, H)} + \sum_{\emptyset \neq \cK \subseteq \cH} (-1)^{\size{\cK}} \size{\Hom(G, \bigcap_{K \in \cK} K)} \\
		&= \sum_{\cK \subseteq \cH} (-1)^{\size{\cK}} \size{\Hom(G, \bigcap_{K \in \cK} K)}.
	\end{align*}
	(Note that if \( \cK = \emptyset \) then \( \bigcap_{K \in \cK} K = \set{ h \in H \given h \in K\ (K \in \cK)} = H \).)
\end{proof}

The next lemma is inspired by the work of Lov\'asz~\cite{Lov67, Lov72}.
\begin{lem}
	\label{lem:hom}
	Let \( G \) and \( H \) be finite groups.
	Then \( G \cong H \) if and only if \( \size{G} = \size{H} \) and
	\begin{equation*}
		\size{\Hom(G, K)} = \size{\Hom(H, K)}
	\end{equation*}
	for every subgroup~\( K \) of~\( G \).
\end{lem}

\begin{proof}
	From \( \size{\Hom(G, K)} = \size{\Hom(H, K)} \) for every subgroup~\( K \) of~\( G \),
	it follows that \( \size{\Epi(H, G)} = \size{\Epi(G, G)} \ge 1 \) by \cref{lem:epi}.
	Hence, we have an epimorphism from~\( H \) to~\( G \).
	It must be an isomorphism because \( \size{G} = \size{H} \).
\end{proof}

For a unital \( R \)-algebra~\( A \) over a unital commutative ring~\( R \),
the unit group of~\( A \) is denoted by~\( A^\ast \).
The following adjoint is well-known.
\begin{lem}
	\label{lem:adjoint}
	Let \( R \) be a unital commutative ring, \( G \) a group and \( A \) a unital \( R \)-algebra.
	Then there is a bijection
	\begin{equation*}
		\Hom(RG, A) \cong \Hom(G, A^\ast)
	\end{equation*}
	which is natural in \( G \) and \( A \).
	\textup{(}Namely, the group algebra functor is left adjoint to the unit group functor.\textup{)}
\end{lem}

\begin{proof}
	The restriction \( (f \colon RG \to A) \mapsto (f|_G \colon G \to A^\ast) \) gives rise to the desired bijection;
	See \cite[pp. 204--205, 490]{Row08} for details.
\end{proof}

Now we propose a definition of hereditary groups which is crucial in this study.
\begin{dfn}
	\label{dfn:hereditary group}
	Set
	\(
		M = \set{ [G] \given \text{\( G \) is a finite group} }
	\)
	where the symbol~\( [G] \) denotes the isomorphism class of a group~\( G \).
	It becomes a commutative monoid with an operation
	\begin{equation*}
		[G] + [H] = [G \times H].
	\end{equation*}
	Let \( K(M) \) denote the Grothendieck group\footnote{This is also called the \emph{group completion} of \( M \). See \cite[Theorem~1.1.3]{Ros94}.} of~\( M \).
	As it is a \ZZ-module (abelian group), we can extend scalars and obtain a \QQ-vector space
	\(
		L(M) = \QQ \otimes_\ZZ K(M).
	\)
	For a finite unital commutative ring \( R \), define the \QQ-subspace \( S(R) \) of \( L(M) \) by
	\begin{equation*}
		S(R) = \sum_{\substack{\text{\( A \) is a finite}\\ \text{unital \( R \)-algebra}}} \QQ[A^\ast].
	\end{equation*}
	We call a finite group~\( G \) is a \emph{hereditary group} over~\( R \) if \( [K] \in S(R) \) for every subgroup~\( K \) of~\( G \).
\end{dfn}

From the definition, being hereditary group is a subgroup-closed property.
Note that the group completion \( M \to K(M) \) is injective because \( M \) is cancellative by the Krull-Schmidt theorem;
The localization \( K(M) \to L(M) \) is also injective because \( K(M) \) is torsion-free.
Thus, \( M \) can be identified with a submonoid of~\( L(M) \).

\begin{exm}
	Let \( C_{q} \) denote the cyclic group of order \( q \).
	Then from
	\( \FF_5^\ast \cong C_4 \) and \( (\FF_5C_5)^\ast \cong C_4 \times (C_5)^4 \)
	we have
	\begin{equation*}
		[C_5] = \tfrac{1}{4}[(\FF_5C_5)^\ast] - \tfrac{1}{4}[\FF_5^\ast] \in S(\FF_5).
	\end{equation*}
	In particular, \( C_5 \) is a hereditary group over \( \FF_5 \).
\end{exm}

The next criterion---our main result---shows that hereditary groups are determined by their group algebras.

\begin{cri}
	\label{cri:iso}
	Let \( G \) and \( H \) be finite groups and let \( R \) be a finite unital commutative ring.
	Suppose \( G \) is a hereditary group over~\( R \).
	If \( RG \cong RH \) then \( G \cong H \).
\end{cri}

The proof is done by describing the number of group homomorphisms in terms of the number of unital algebra homomorphisms.

\begin{proof}[Proof of \cref{cri:iso}]
	Since a unital commutative ring~\( R \) has invariant basis number (IBN) property,
	we have \( \size{G} = \size{H} \) from \( RG \cong RH \).
	Hence, by \cref{lem:hom}, it suffices to prove that
	\( \size{\Hom(G, K)} = \size{\Hom(H, K)} \)
	for every subgroup~\( K \) of~\( G \).

	Since \( G \) is a hereditary group over~\( R \), we have \( [K] \in S(R) \).
	Therefore, there is a positive integer \( n \) and finite unital \( R \)-algebras \( A, B \) such that
	\begin{equation*}
		[K] = \tfrac{1}{n}[A^\ast] - \tfrac{1}{n}[B^\ast].
	\end{equation*}
	Namely,
	\(
		A^\ast \cong B^\ast \times K^n
	\).
	Hence,
	\begin{equation*}
			\size{\Hom(G, A^\ast)}
			= \size{\Hom(G, B^\ast \times K^n)}
			= \size{\Hom(G, B^\ast)} \times \size{\Hom(G, K)}^n.
	\end{equation*}
	Thus, by \cref{lem:adjoint}, we can obtain
	\begin{equation*}
		\size{\Hom(G, K)} =
		\left( \frac{\size{\Hom(G, A^\ast)}}{\size{\Hom(G, B^\ast)}} \right)^{1/n}
		=
		\left( \frac{\size{\Hom(RG, A)}}{\size{\Hom(RG, B)}} \right)^{1/n}.
	\end{equation*}
	We can calculate \size{\Hom(H, K)} similarly and conclude that \( \size{\Hom(G, K)} = \size{\Hom(H, K)} \) from \( RG \cong RH \).
\end{proof}

\begin{rmk}
	Studying a finite group that is a `linear combination' of unit groups, precisely an element of~\( S(R) \),  is essential because
	even the cyclic group of order five cannot be realized as a unit group of any unital ring.
	(See a theorem by Davis and Occhipinti~\cite[Corollary~3]{DO14}, for example.)
	This is quite different from the fact that every abelian \( p \)-groups are hereditary groups over~\( \FF_p \) (\cref{lem:abelian p-groups}).
\end{rmk}

Using this criterion, we provide new proofs for some early theorems on the modular isomorphism problem in the last section.

\section{Quasi-regular groups}
We show that, with \cref{cri:iso},  study of quasi-regular groups can be used to study the isomorphism problem.
Throughout this section, \( R \) denotes a unital commutative ring.

\begin{dfn}
	Let \( A \) be an \( R \)-algebra.
	Define the \emph{quasi-multiplication} on~\( A \) by
	\begin{equation*}
		x \circ y = x + y + xy.
	\end{equation*}
	An element~\( x \in A \) is called \emph{quasi-regular} if there is an element~\( y \in A \) such that
	\( x \circ y = 0 = y \circ x \).
	We denote the set of all quasi-regular elements by~\( Q(A) \).
	It forms a group under the quasi-multiplication and we call it the \emph{quasi-regular group} of~\( A \).
	If \( A = Q(A) \) then \( A \) is called \emph{quasi-regular} (or \emph{radical}).
\end{dfn}

Quasi-multiplication is also called \emph{circle operation} or \emph{adjoint operation}.
Accordingly, quasi-regular groups are also called \emph{circle groups} or \emph{adjoint groups}.
As these terms have completely different meaning in other contexts, we avoid using them.

If an \( R \)-algebra~\( A \) has a multiplicative identity then there is an isomorphism \( Q(A) \to A^\ast \) that is defined by~\( x \mapsto 1 + x \).
We study how quasi-regular groups are related to unit groups, especially when algebras do not have multiplicative identities, in the rest of this section.

\begin{dfn}
	Let \( A \) be an \( R \)-algebra.
	We denote the \emph{unitization} of~\( A \) by~\( A^\un \):
	it is a direct product \( A^\un = A \times R \) as \( R \)-modules and its multiplication is defined by
	\begin{equation*}
		(x, r) \times (y, s) = (sx + ry + xy, rs).
	\end{equation*}
\end{dfn}

Note that \( (0, 1) \in A^\un \) is the multiplicative identity.
The unitization is also called the \emph{Dorroh extension}, especially the case~\( R = \ZZ \).

\begin{lem}
	\label{lem:unit}
	Let \( A \) be a quasi-regular \( R \)-algebra.
	Then an element \( (x, r) \in A^\un \) is a unit if and only if \( r \in R \) is a unit.
\end{lem}

\begin{proof}
	The `only if' part is trivial.
	Let us assume \( r \in R^\ast \) to show \( (x, r) \in (A^\un)^\ast \).
	As \( r^{-1} \in R \) exists and \( r^{-1}x \in A \), there is an element~\( y \in A \) such that
	\begin{equation*}
		(r^{-1}x) \circ y = 0 = y \circ (r^{-1}x)
	\end{equation*}
	because \( A \) is quasi-regular.
	Then it can be shown that \( (x, r)^{-1} = (r^{-1}y, r^{-1}) \) by direct calculation.
\end{proof}

\begin{lem}
	\label{lem:quasi-regular}
	Let \( A \) be a quasi-regular \( R \)-algebra.
	Then
	\begin{equation*}
		(A^\un)^\ast \cong Q(A) \times R^\ast.
	\end{equation*}
	In particular, \( [Q(A)] \in S(R) \) if \( R \) and \( A \) are finite.
\end{lem}

\begin{proof}
	Note that there are homomorphisms
	\begin{equation*}
		\begin{tikzcd}[row sep=tiny]
			Q(A)    & (A^\un)^\ast \arrow[l] \arrow[r]             & R^\ast \\
			r^{-1}x & (x, r) \arrow[r, maps to] \arrow[l, maps to] & r
		\end{tikzcd}
	\end{equation*}
	which are well-defined by \cref{lem:unit}.
	It is straightforward to check that these satisfy the universal property of a direct product.
\end{proof}

\begin{rmk}
	For determining whether a finite group is a quasi-regular group of an \( R \)-algebra,
	a theorem by Sandling~\cite[Theorem~1.7]{San74} would be helpful.
	It should be noted that quasi-regular groups also have severely restricted structure as unit groups.
	See~\cite[p.~343]{San84}.
\end{rmk}

\section{Modular isomorphism problem}
As application of \cref{cri:iso},
we provide new proofs for theorems by Deskins~\cite{Des56} and Passi-Sehgal~\cite{PS72}
which are early theorems on the modular isomorphism problem.

\subsection{Abelian (Class at most one)}
Let us state a well-known theorem by Deskins~\cite[Theorem~2]{Des56}.

\begin{thm}[Deskins]
	\label{thm:Deskins}
	Let \( G \) and \( H \) be finite \( p \)-groups.
	Suppose \( G \) is abelian.
	If \( \FF_pG \cong \FF_pH \) then \( G \cong H \).
\end{thm}

To use our criterion, we need to prove the following, which is also useful to prove \cref{thm:Passi-Sehgal-odd,thm:Passi-Sehgal-even}.

\begin{lem}
	\label{lem:abelian p-groups}
	Finite abelian \( p \)-groups are hereditary groups over~\( \FF_p \).
\end{lem}

\begin{proof}
	Let \( C_{p^n} \) denote the cyclic \( p \)-group of order~\( p^n \).
	Since a finite abelian \( p \)-group is a direct product of finite cyclic \( p \)-groups and
	being finite cyclic \( p \)-group is a subgroup-closed property, it suffices to prove that \( [C_{p^n}] \in S(\FF_p) \).
	We prove it by induction on~\( n \).

	\paragraph*{Base case (\( n = 0 \))}
	Clear by \cref{dfn:hereditary group}.

	\paragraph*{Inductive case (\( n > 0 \))}
	Let \( \Delta(C_{p^n}) \) denote the augmentation ideal of~\( \FF_pC_{p^n} \).
	Then
	\begin{equation*}
		(\FF_pC_{p^n})^\ast = V \times \FF_p^\ast
	\end{equation*}
	where~\( V = 1 + \Delta(C_{p^n}) \).
	It is clear that \( V \) is a finite abelian group.
	Furthermore, it is easy to see the exponent of~\( V \) equals~\( p^n \).
	Therefore,
	\(
		V \cong \prod_{i = 1}^n (C_{p^i})^{a_i}
	\)
	for some non-negative integers \( a_1, \dotsc, a_{n-1} \) and a positive integer \( a_n \).
	Hence,
	\begin{align*}
		[C_{p^n}]
		&= \frac{1}{a_n}[V] - \sum_{i = 1}^{n - 1} \frac{a_i}{a_n}[C_{p^i}] \\
		&= \frac{1}{a_n}[(\FF_pC_{p^n})^\ast] - \frac{1}{a_n}[\FF_p^\ast] - \sum_{i = 1}^{n - 1} \frac{a_i}{a_n}[C_{p^i}]
	\end{align*}
	and we have \( [C_{p^n}] \in S(\FF_p) \) by induction.

	(This lemma can be also proved by appealing to the structure theorem of~\( (\FF_pC_{p^n})^\ast \) by Janusz~\cite[Theorem~3.1]{Jan70}.)
\end{proof}

With this lemma, we provide a new proof of the Deskins theorem.

\begin{proof}[Proof of \cref{thm:Deskins}]
	Since the finite abelian \( p \)-group \( G \) is a hereditary group over~\( \FF_p \) by \cref{lem:abelian p-groups}, \( \FF_pG \cong \FF_pH \) implies \( G \cong H \) by \cref{cri:iso}.
\end{proof}

\begin{rmk}
	Besides the original proof by Deskins,
	an alternative simple proof is given by Coleman~\cite[Theorem~4]{Col64}.
	Proofs can be found in monographs such as \cite[2.4.3]{KP69}, \cite[Lemma~14.2.7]{Pas77}, \cite[(III.6.2)]{Seh78}, \cite[Theorem~9.6.1]{PMS02} or \cite[Theorem~4.10]{Seh03} as well.
\end{rmk}

\subsection{Class two and exponent~$p$}
The aim of this subsection is to provide a new proof of the following theorem by Passi and Sehgal~\cite[Corollary~13]{PS72}.
\begin{thm}[Passi-Sehgal]
	\label{thm:Passi-Sehgal-odd}
	Let \( G \) and \( H \) be finite \( p \)-groups.
	Suppose \( G \) is of class two and exponent~\( p \).
	If \( \FF_pG \cong \FF_pH \) then \( G \cong H \).
\end{thm}

See also \cref{rmk:M3}.
To use our criterion, we need to prove the following.

\begin{lem}
	\label{lem:exp p}
	Finite \( p \)-groups of class two and exponent~\( p \) are hereditary groups over~\( \FF_p \).
\end{lem}

A key ingredient for the proof is a slight modification of the theorem by Ault and Watters~\cite{AW73, GR73}.

\begin{thm}[Ault-Watters]
	\label{thm:Ault-Watters}
	Let \( G \) be a finite \( p \)-group.
	Suppose \( G \) is of class two and exponent~\( p \).
	Then there is a finite quasi-regular \( \FF_p \)-algebra~\( A \) with~\( Q(A) \cong G \).
\end{thm}

\begin{proof}
	It is proved in \cite{AW73} that there is a finite quasi-regular ring \( A = G \) with operations\footnote{Beware that the operations defined in \cite{AW73} are incorrect; It is corrected in \cite{GR73}.} defined by
	\begin{align*}
		g + h      &= g \cdot h \cdot m(g, h)^{-1} \\
		g \times h &= m(g, h)
	\end{align*}
	and \( Q(A) \cong G \);
	Here \( m \colon G \times G \to \zeta(G) \) is a certain map where~\( \zeta(G) \) is the center of~\( G \).
	In particular, \( 1 \in G \) is the additive identity of~\( A \).
	By induction, it can be shown that
	\begin{equation*}
		\underbrace{g + \dotsb + g}_{n} = g^n \cdot m(g, g)^{-n(n-1)/2}
	\end{equation*}
	for a positive integer~\( n \).
	Since the prime~\( p \) is odd because of our assumption,
	we can prove
	\begin{equation*}
		\underbrace{g + \dotsb + g}_{p} = 1.
	\end{equation*}
	Therefore, there is a canonical \( \FF_p \)-algebra structure on~\( A \).
\end{proof}

\begin{proof}[Proof of \cref{lem:exp p}]
	Since every finite abelian \( p \)-groups are hereditary groups over \( \FF_p \) by \cref{lem:abelian p-groups},
	it suffices to prove that \( [G] \in S(\FF_p) \) for every finite \( p \)-group~\( G \) of class two and exponent~\( p \).
	Because such \( p \)-group~\( G \) is a quasi-regular group of some finite quasi-regular \( \FF_p \)-algebra by \cref{thm:Ault-Watters}, \( [G] \in S(\FF_p) \) follows from \cref{lem:quasi-regular}.
\end{proof}

Now we are in position to prove a theorem by Passi and Sehgal.

\begin{proof}[Proof of \cref{thm:Passi-Sehgal-odd}]
	Since the finite \( p \)-group~\( G \) of class two and exponent~\( p \) is a hereditary group over~\( \FF_p \) by \cref{lem:exp p}, \( \FF_pG \cong \FF_pH \) implies \( G \cong H \) by \cref{cri:iso}.
\end{proof}

\subsection{Class two and exponent four}

The aim of this subsection is to prove the even prime counterpart of \cref{thm:Passi-Sehgal-odd}.

\begin{thm}[Passi-Sehgal]
	\label{thm:Passi-Sehgal-even}
	Let \( G \) and \( H \) be finite \( 2 \)-groups.
	Suppose \( G \) is of class two and exponent four.
	If \( \FF_2G \cong \FF_2H \) then \( G \cong H \).
\end{thm}

\begin{rmk}
	\label{rmk:M3}
	Note that the third dimension subgroup of~\( G \) modulo~\( p \) is \( G^p\gamma_3(G) \) if~\( p \neq 2 \) and \( G^4(G')^2\gamma_3(G) \) if~\( p = 2 \) where~\( \gamma_3(G) \) denotes the third term of the lower central series of~\( G \). 
	Thus, the assumption for a group in \cref{thm:Passi-Sehgal-odd} or \cref{thm:Passi-Sehgal-even} holds if \( G \) has the trivial third modular dimension subgroup.
	Actually, this is how assumption is stated by Passi and Sehgal~\cite[Corollary~7]{PS72}.

	Nowadays more is known.
	A theorem by Sandling~\cite[Theorem~1.2]{San89} provides a positive solution for a finite \( p \)-group of class two with elementary abelian commutator subgroup.
\end{rmk}

A strategy for the proof is the same as \cref{thm:Passi-Sehgal-odd}.
The even prime counterpart of the Ault-Watters theorem is the following theorem by Bovdi~\cite{Bov96}.

\begin{thm}[Bovdi]
	\label{thm:Bovdi}
	Let \( G \) be a finite \( 2 \)-group.
	Suppose \( G \) is of class two and exponent four.
	Then there is a finite quasi-regular \( \FF_2 \)-algebra~\( A \) with \( Q(A) \cong G \).
\end{thm}

\begin{lem}
	\label{lem:exp 4}
	Every \( 2 \)-group of class two and exponent four is a hereditary group over \( \FF_2 \).
\end{lem}

\begin{proof}
	Since every finite abelian \( 2 \)-groups are hereditary groups over \( \FF_2 \) by \cref{lem:abelian p-groups},
	it suffices to prove that \( [G] \in S(\FF_2) \) for every finite \( 2 \)-group~\( G \) of class two and exponent four.
	Because such \( 2 \)-group~\( G \) is a quasi-regular group of some finite quasi-regular \( \FF_2 \)-algebra by \cref{thm:Bovdi}, \( [G] \in S(\FF_2) \) follows from \cref{lem:quasi-regular}.
\end{proof}
\begin{proof}[Proof of \cref{thm:Passi-Sehgal-even}]
	Since the finite \( 2 \)-group \( G \) of class two and exponent four is a hereditary group over \( \FF_2 \) by \cref{lem:exp 4},
	\( \FF_2G \cong \FF_2H \) implies \( G \cong H \) by \cref{cri:iso}.
\end{proof}

\end{document}